\documentclass{article}

\usepackage{amssymb} 
\usepackage{amsmath} 
\usepackage{latexsym} 

\usepackage[dvips]{color,psfrag,graphicx}

\newtheorem{theorem}{Theorem}[section]

\newtheorem{corollary}[theorem]{Corollary}
\newtheorem{proposition}[theorem]{Proposition}

\newtheorem{definition}[theorem]{Definition}

\newenvironment{proof}{{\par\addvspace{0.1cm}\noindent \bf Proof. }}{\hfill$\Box$\par\medskip} 

\newtheorem{remark}[theorem]{Remark}

\numberwithin{equation}{section}

\def\vf{\varphi}

\begin{document}

\title{The configuration space of equilateral and equiangular hexagons}

\author{Jun O'Hara}
%
%
\maketitle

\begin{abstract}
We study the configuration space of equilateral and equiangular spatial hexagons for any bond angle by giving explicit expressions of all the possible shapes. We show that the chair configuration is isolated, whereas the boat configuration allows one-dimensional deformations which form a circle in the configuration space. 
\end{abstract}



\section{Introduction} 

Let $\mathcal P$ be a polygon with $n$ vertices in $\mathbb R^3$. 
We express $\mathcal P$ by its vertices, $\mathcal P=(P_0, P_1, \cdots, P_{n-1})$, with suffixes modulo $n$. 
A polygon $\mathcal P$ is called {\em equilateral} if the edge length $|P_{i+1}-P_i|$ is constant, and {\em equiangular} if the angle $\angle P_{i-1}P_iP_{i+1}$ is constant. 
This angle between two adjacent edges is called the {\em bond angle} and will be denote by $\theta$ in this paper. 
An equilateral and equiangular polygon can be considered as a mathematical model of a cycloalkane. 
We are interested in the set of all the possible shapes (which are called conformations in chemistry) of cycloalkanes when the number $n$ of carbons and the bond angle $\theta$ is fixed, i.e. in the language of mathematics, the configuration space of equilateral and equiangular polygons. 

Gordon Crippen studied it for $n\le 7$ (\cite{C}). 
To be precise, what he obtained is not the configuration space itself, but the space of the ``{\sl metric matrices}'', which are $n\times n$ matrices whose entries are inner products of pairs of edge vectors, and then he gave the corresponding conformations. 
When $n=4$ (cyclobutanes) and $n=5$ (cyclopentanes) he considered all the possible bond angles, 
%
but when $n=6$ (cyclohexanes) and $n=7$ (cyclopentanes) he fixed the bond angle to be the ideal tetrahedral bond angle $\cos^{-1}(-\frac13)\approx 109.47^\circ$ (Figure \ref{tetrahedron}). 
He showed that if $n=6$ the conformation space is a union of a circle which contains a {\em boat} (Figure \ref{boat6}) and an isolated point of a {\em chair} (Figure \ref{chair6}), and that if $n=7$ it consists of two circles, one for boat/twist-boat and the other for chair/twist-chair. 
In these two cases, he showed it by searching out all the possible values of the entries of the metric matrix through numerical experiment with 0.05 step size. 

\begin{figure}[htbp]
\begin{center}
\begin{minipage}{.28\linewidth}
\begin{center}
\includegraphics[width=.9\linewidth]{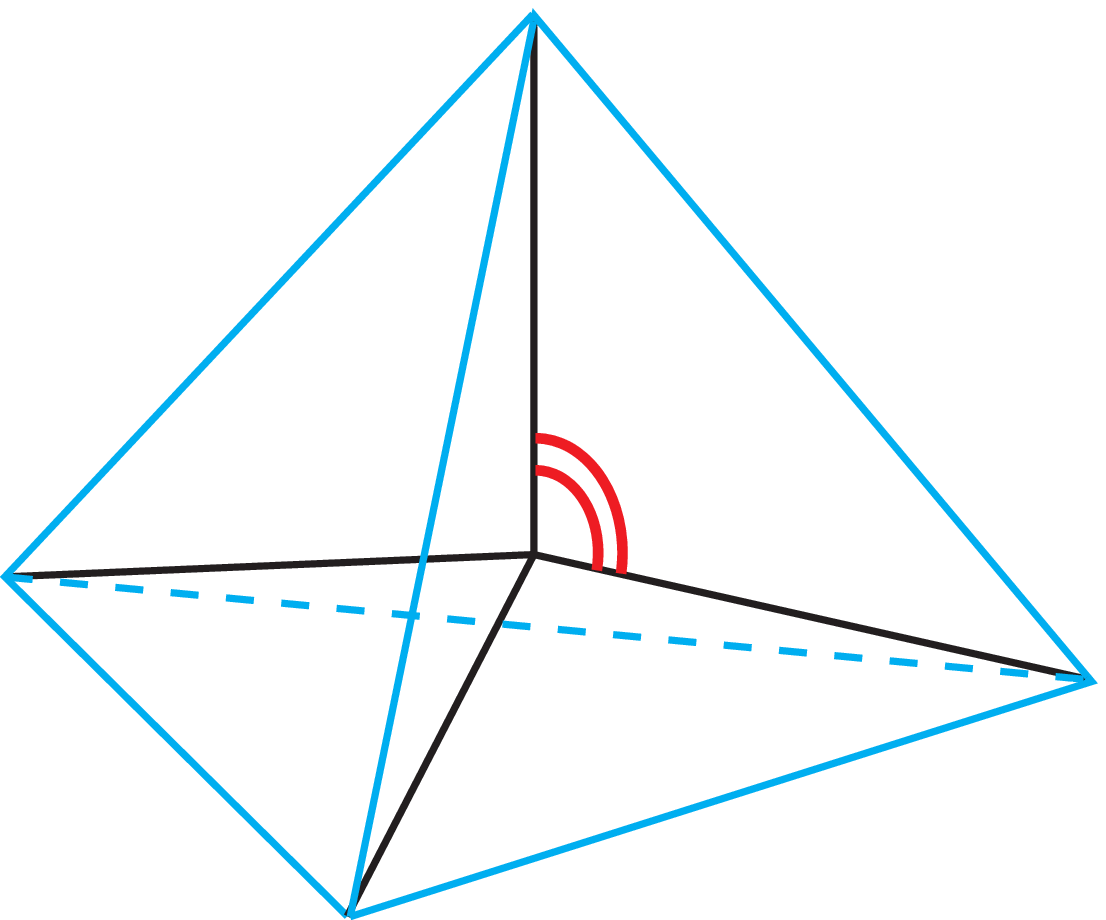}
\caption{$\theta=\cos^{-1}(-\frac13)$}
\label{tetrahedron}
\end{center}
\end{minipage}
\hskip 0.5cm
\begin{minipage}{.25\linewidth}
\begin{center}
\includegraphics[width=\linewidth]{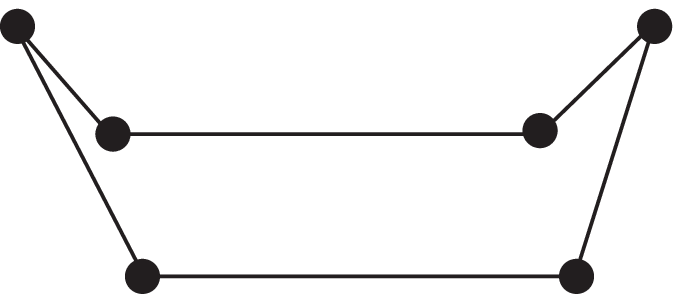}
\caption{{\em boat}}
\label{boat6}
\end{center}
\end{minipage}
\hskip 0.5cm
\begin{minipage}{.25\linewidth}
\begin{center}
\includegraphics[width=\linewidth]{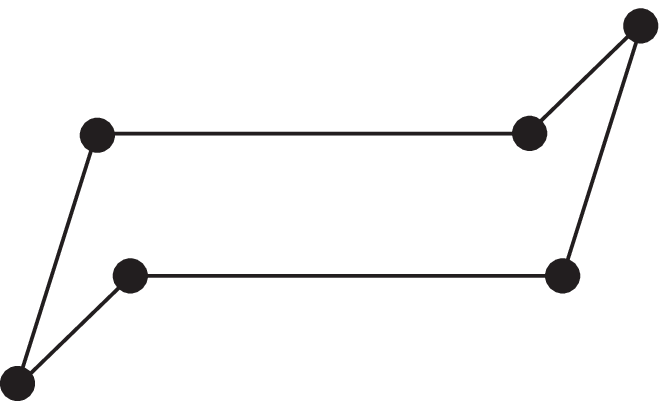}
\caption{{\em chair}}
\label{chair6}
\end{center}
\end{minipage}
\end{center}
\end{figure}

In this paper we study the configuration space of hexagons for any bond angle. 
We give explicit expressions of all the possible configurations in terms of our parameters illustrated in Figure \ref{double-cone}. 
We show that the topological type of the configuration space depends on the bond angle $\theta$. 

If $\theta$ is big ($\frac\pi3<\theta<\frac{2\pi}3$) the situation is same as that of cyclohexanes of ideal tetrahedral bond angle studied by Crippen. 
On the other hand, if $\theta$ is small ($0<\theta<\frac{\pi}3$), a new configuration (``{\sl inward crown}'' illustrated in Figure \ref{cylinder}) appears, and the one dimensional continuum of deformation of a boat is divided into two pieces, which implies that we cannot deform a boat into its mirror image. We remark that we distinguish vertices in our study. Hence our configuration space is not equal to the space of shapes. 
We have an exceptional case if $\theta=\frac\pi3$, when the inward crown, which degenerates to a doubly covered triangle, can be deformed to boats via newly appeared families of configurations. 

In any case, a chair is an isolated configuration, whereas a boat allows one-dimensional deformations starting from and ending at it. 
Moving pictures of deformation of a boat can be found at 

\verb+http://www.comp.tmu.ac.jp/knotNRG/math/configuration.html+

Of course, as extremal cases, we have a $6$ times covered multple edge and a regular hexagon when $\theta=0,\frac{2\pi}3$. 

\medskip
There have been a great number of studies of the configuration spaces of {\em linkages}. 
An excellent survey can be found in \cite{D}. 
If we drop the condition of being equiangular, it is known that the space of equilateral polygons has a symplectic structure (\cite{K-M}). 

\medskip
This paper is based on the author's talk at ``Knots and soft-matter physics, Topology of polymers and related topics in physics, mathematics and biology'', YITP, Kyoto, 2008 ({A short announcement of the result without proof was reported in {\sl Bussei Kenkyuu} vol. 92 (2009) 119\,--\,122}). 

\medskip 
{\bf Notations}. Throughout the paper, we agree that $C=\cos(\frac\theta2)$ and $S=\sin(\frac\theta2)$. 
The suffixes are understood modulo $n$. 
The angle $\angle P_i$ means $\angle P_{i-1}P_iP_{i+1}$. 

\section{Preliminaries}

\begin{definition} \rm 
Put 
\[\displaystyle \widetilde{\mathcal{M}}_n(\theta)=\left\{\begin{array}{c}
\mathcal{P}=(P_0, \ldots, P_{n-1})\\
(P_i\in\mathbb{R}^3)
\end{array}
\left|\begin{array}{l}
|P_i-P_{i+1}|=1, \\[0.5mm] 
\angle P_{i-1}P_iP_{i+1}=\theta \end{array} \hspace{0.2cm}(\forall i \>\>(\mbox{mod.}\> n)) \right\}.\right.\]
Let $G$ be 
the group of orientation preserving isometries of $\mathbb{R}^3$. 
%
Put $\mathcal{M}_n(\theta)=\widetilde{\mathcal{M}}_n(\theta)/G$, and call it the {\em configuration space of $\theta$-equiangular unit equilateral $n$-gons} $(0\le\theta<\pi)$. 
Let us denote the equivalence class of a plolygon $\mathcal{P}$ by $[\mathcal{P}]$. 
\end{definition}

\begin{remark}\rm 
\begin{enumerate}
\item We allow intersections of edges and overlapping of vertices. 
\item We distinguish the vertices when we consider our configuration space. Therefore, two {configurations} illustrated in Figures \ref{chair6} correspond to different points in $\mathcal{M}_6$, although their {\sl shapes} are the same. 
\item When we express an equilateral and equiangular polygon we may fix the first three vertices, $P_0, P_1$, and $P_2$. 
There are $(n-3)$ more vertices, whereas we have $(n-2)$ conditions for the lengths of the edges, and $(n-1)$ conditions for the angles. Therefore, we may expect that the dimension of $\mathcal{M}_n(\theta)$ is equal to $3(n-3)-(n-2)-(n-1)=n-6$ in general if the conditions are independent, which is not the case when $n\le6$ as we will see. 
\end{enumerate}
\end{remark}

\begin{figure}[htbp]
\begin{center}
\begin{minipage}{.45\linewidth}
\begin{center}
\includegraphics[width=.3\linewidth]{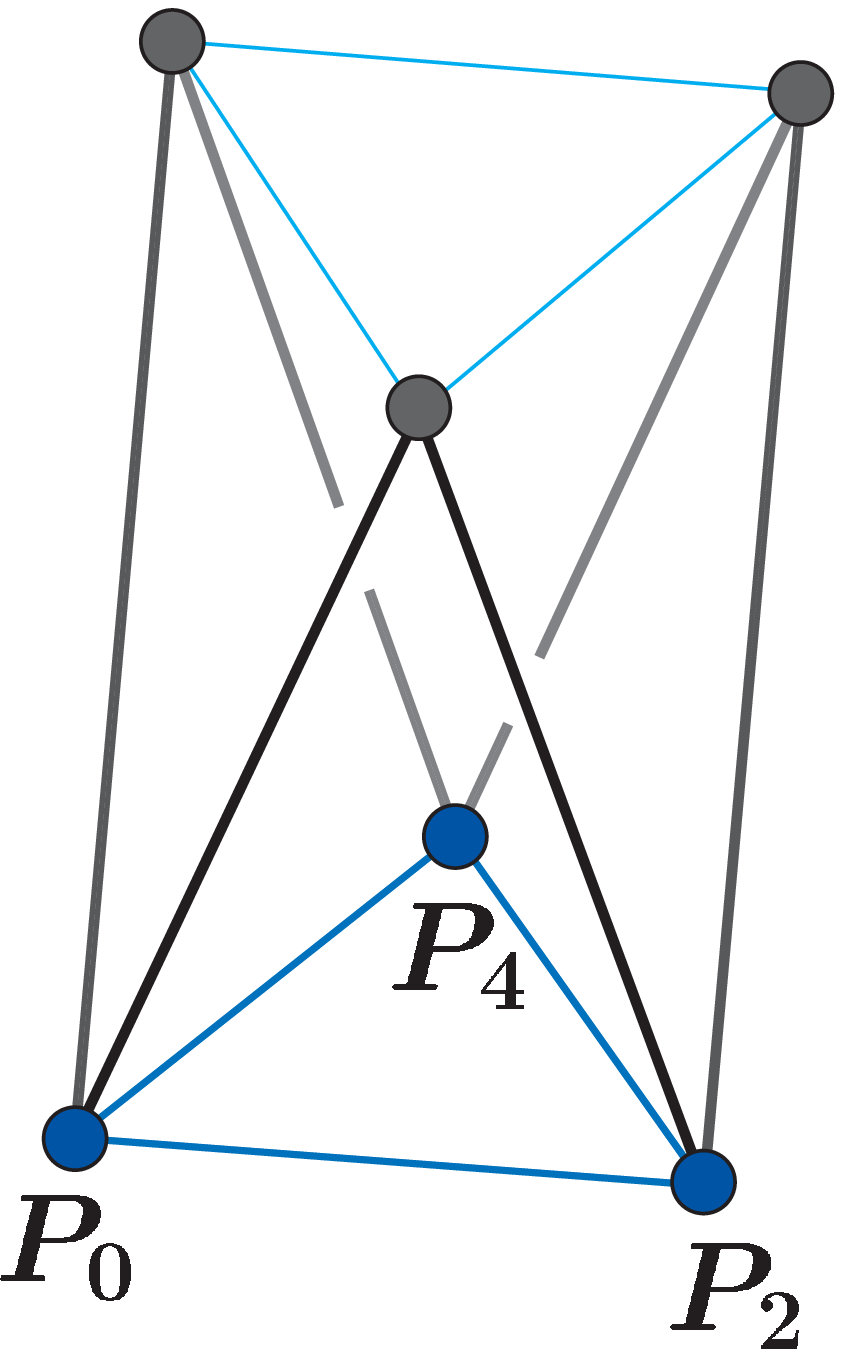}
\end{center}
\end{minipage}
\hskip 0.1cm
\begin{minipage}{.45\linewidth}
\begin{center}
\includegraphics[width=.3\linewidth]{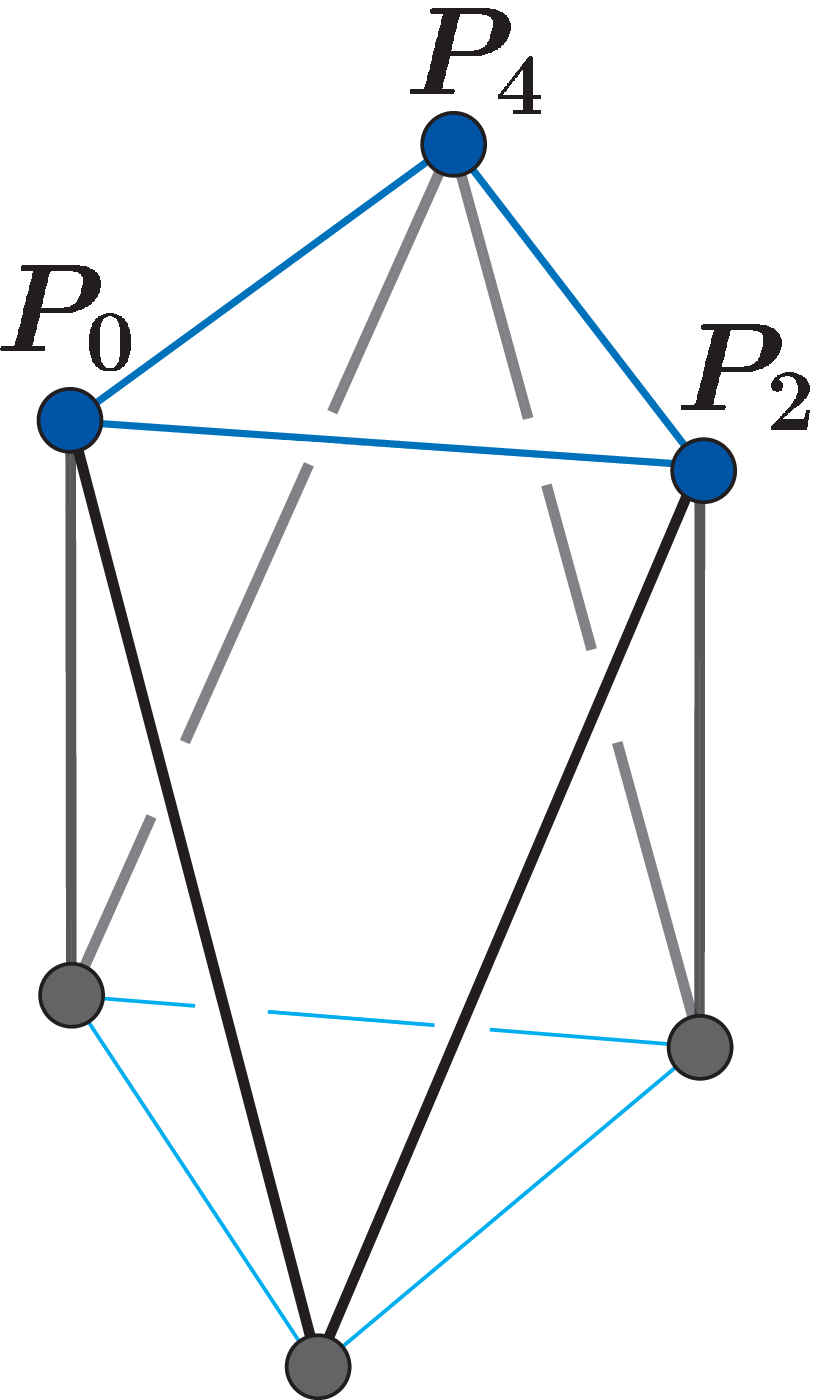}
\end{center}
\end{minipage}
\caption{A chair when $\theta$ is small and its mirror image}
\label{udchair6}
\end{center}
\end{figure}


When $n\le5$ the configuration space is given as follows. 

\begin{proposition} {\rm (\cite{C})} 
The configuration spaces $\mathcal{M}_n(\theta)$ of equilateral and $\theta$-equiangular $n$-gons $(n=3,4,5)$ and the shapes of polygons which correspond to the elements are given by the following. 
\[\begin{array}{l}
\mathcal M_3(\theta)\cong\left\{\hspace{-0.1cm}
\begin{array}{lll}
\{\mbox{\rm $1$ point}\} & (\theta=\frac{\pi}3) & \mbox{\rm regular triangle} \\
\>\emptyset & \mbox{\rm otherwise} &  \\
\end{array}
\right.\\[5mm]
\mathcal M_4(\theta)\cong\left\{\hspace{-0.1cm}
\begin{array}{lll}
\{\mbox{\rm $1$ point}\} & (\theta=0) & \mbox{\rm $4$-folded edge (Figure \ref{n=4_case} left)} \\
\{\mbox{\rm $2$ points}\} & (0<\theta<\frac{\pi}4) & \mbox{\rm folded rhombus (Figure \ref{n=4_case} center) and its mirror image}  \\
\{\mbox{$1$ point}\} & (\theta=\frac{\pi}4) & \mbox{\rm square (Figure \ref{n=4_case} right)}  \\
\>\emptyset &  \mbox{\rm otherwise,} & \\
\end{array}
\right.\\[9mm]
\mathcal M_5(\theta)\cong\left\{\hspace{-0.1cm}
\begin{array}{lll}
\{\mbox{\rm $1$ point}\} & (\theta=\frac{\pi}5) & \mbox{\rm regular star shape (Figure \ref{star-pentagon})}  \\
\{\mbox{\rm $1$ point}\} & (\theta=\frac{3\pi}5) & \mbox{\rm regular pentagon (Figure \ref{regular_pentagon})}  \\
\>\emptyset &  \mbox{\rm otherwise.} & \\
\end{array}
\right.
\end{array}
\]
\begin{figure}[htbp]
\begin{center}
\begin{minipage}{.3\linewidth}
\begin{center}
\includegraphics[width=0.486\linewidth]{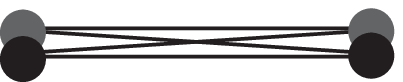}
\end{center}
\end{minipage}
\hskip 0.2cm
\begin{minipage}{.3\linewidth}
\begin{center}
\includegraphics[width=0.648\linewidth]{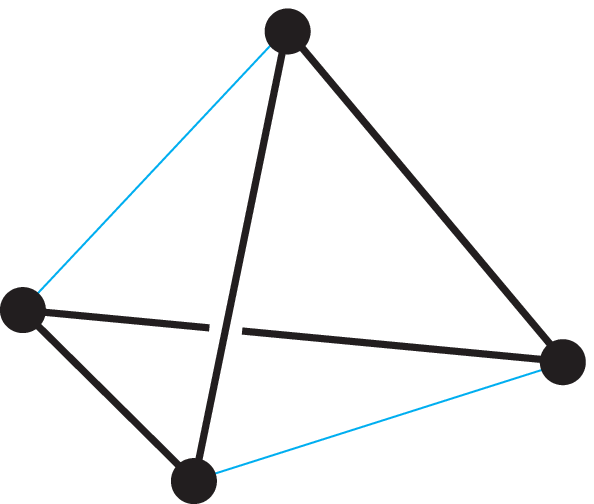}
\end{center}
\end{minipage}
\hskip 0.2cm
\begin{minipage}{.3\linewidth}
\begin{center}
\includegraphics[width=0.48\linewidth]{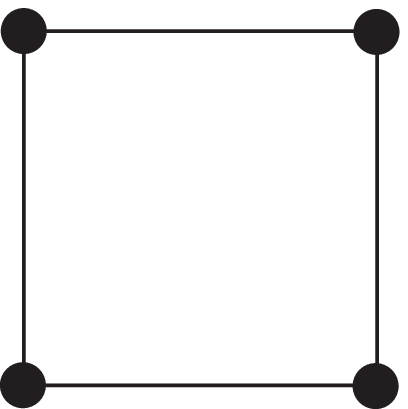}
\end{center}
\end{minipage}
\caption{$n=4$ case. The middle is a non-planar configuration.}
\label{n=4_case}
\end{center}
\end{figure}
%
\begin{figure}[htbp]
\begin{center}
\begin{minipage}{.3\linewidth}
\begin{center}
\includegraphics[width=.4126\linewidth]{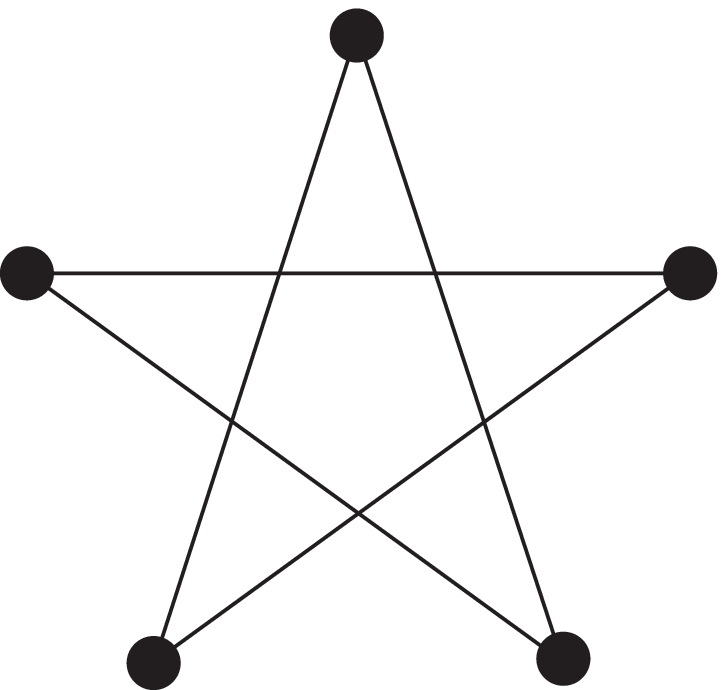}
\caption{Regula star shape}
\label{star-pentagon}
\end{center}
\end{minipage}
\hskip 0.4cm
\begin{minipage}{.3\linewidth}
\begin{center}
\includegraphics[width=.667\linewidth]{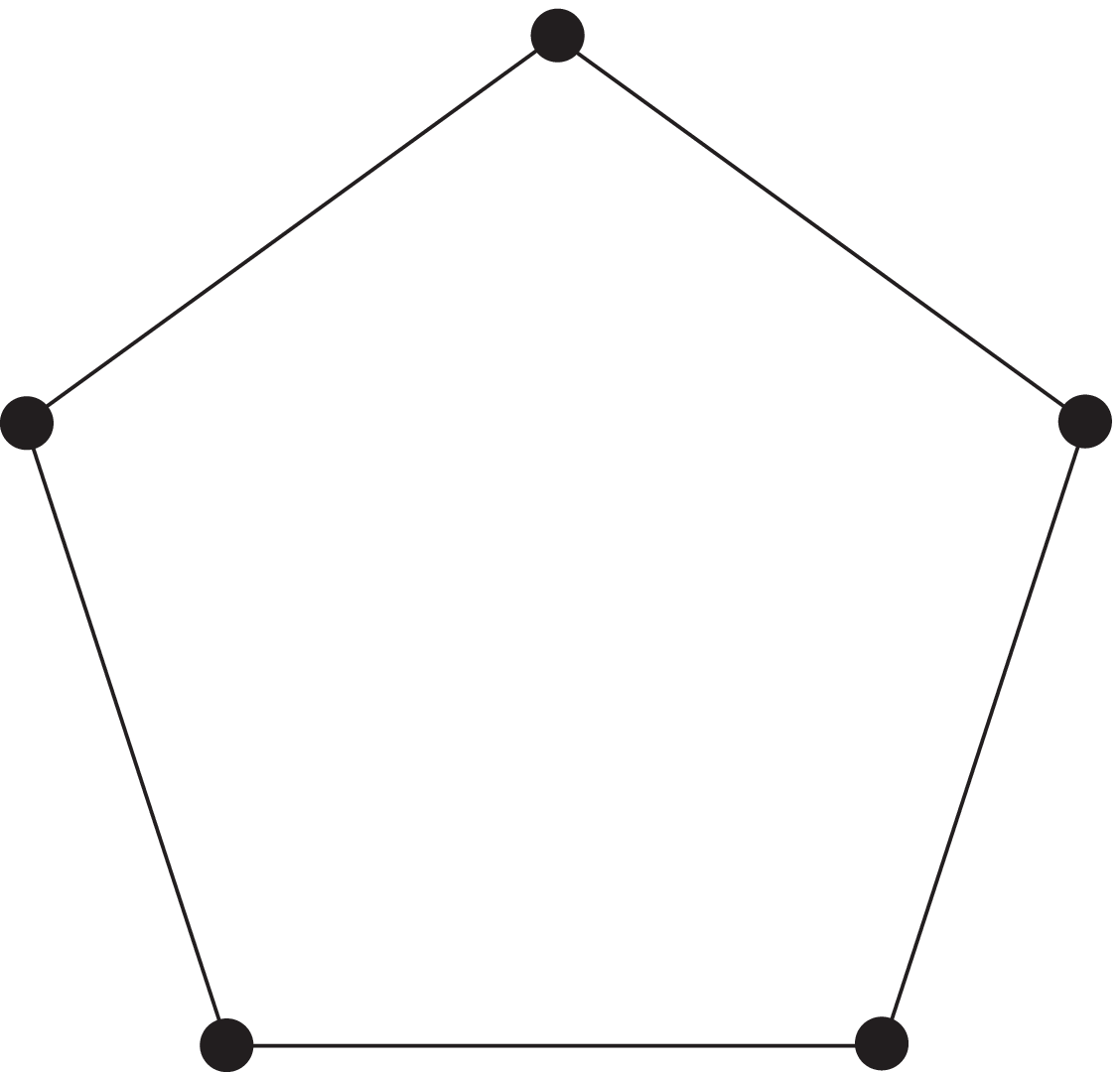}
\caption{Regular pentagon}
\label{regular_pentagon}
\end{center}
\end{minipage}
\end{center}
\end{figure}
\end{proposition}
\begin{proof}
The cases when $n=3,4$ are obvious. 

Suppose $n=5$. Let $(P_0,\ldots, P_4)\in\widetilde{\mathcal{M}}_5(\theta)$. Then $\theta\ne0,\pi$. 
We may assume, after a motion of $\mathbb{R}^3$, that $P_0=(0,0,0), \,P_1=(1,0,0)$, and $P_4=(\cos\theta, \sin\theta, 0)$. 
Then $P_2$ and $P_3$ can be expressed by 
$$
P_2=\left(\begin{array}{c}
1-\cos\theta\\ \sin\theta\cos\vf_2\\ \sin\theta\sin\vf_2
\end{array}\right), 
\>
P_3=\left(\begin{array}{ccc}
\cos\theta & \sin\theta & 0\\
\sin\theta & -\cos\theta & 0\\
0 & 0 & 1
\end{array}
\right)
\left(\begin{array}{c}
1-\cos\theta\\ \sin\theta\cos\vf_3\\ \sin\theta\sin\vf_3\
\end{array}\right)
$$
for some $\vf_2$ and $\vf_3$. 
Using symmetry in a plane that contains the angle bisector of $\angle P_1P_0P_4$ and the $z$-axis, we can deduce from $|P_2P_4|^2=|P_1P_3|^2=\left(2\sin\frac\theta2\right)^2$ that 
$$
\cos\vf_2=\cos\vf_3=\frac{1-2\cos\theta+2\cos^2\theta}{2\sin^2\theta},
$$
which implies $\sin\vf_2=\pm\sin\vf_3$. 
Then $\overrightarrow{P_1P_2}\cdot \overrightarrow{P_2P_3}=-\cos\theta$ implies 
\begin{equation}\label{f_n5}
\frac{-4\cos^2\theta+2\cos\theta+1}{4(1-\cos\theta)}=\sin^2\theta\sin\vf_2(\sin\vf_2-\sin\vf_3). 
\end{equation}
It follows that, whether $\sin\vf_2=\sin\vf_3$ or $\sin\vf_2=-\sin\vf_3$, $\cos\theta=\frac{1\pm\sqrt5}4$, namely, $\theta=\frac\pi5,\frac{3\pi}5$, and $\vf_2=\vf_3=0$, which means only regular star shape and regular pentagon can appear as equilateral and equiangular pentagons. 
\end{proof}
%

\section{Equilateral and equiangular hexagons} 

Put $C=\cos(\frac\theta2)$ and $S=\sin(\frac\theta2)$. 

First note that $P_0, P_2$, and $P_4$ form a regular triangle of edge length $2S$. We may fix 
\begin{equation}\label{P024}
P_0=\left(\!\begin{array}{c}
-S \\ 0 \\ 0
\end{array}\!\right), \>
P_2=\left(\!\begin{array}{c}
S \\ 0 \\ 0
\end{array}\!\right), \>
P_4=\left(\!\!\begin{array}{c}
0 \\ \sqrt3S \\ 0
\end{array}\!\!\right). 
\end{equation}
Any element in $\mathcal{M}_6(\theta)$ has exactly one representative hexagon with $P_0, P_2$, and $P_4$ being as above. 
Let us use a ``{\sl double cone}''or suspension expression of a hexagon (Figure \ref{double-cone}), namely, we express $P_1, P_3$, and $P_5$ by 
\begin{figure}[htbp]
\begin{center}
\includegraphics[width=.35\linewidth]{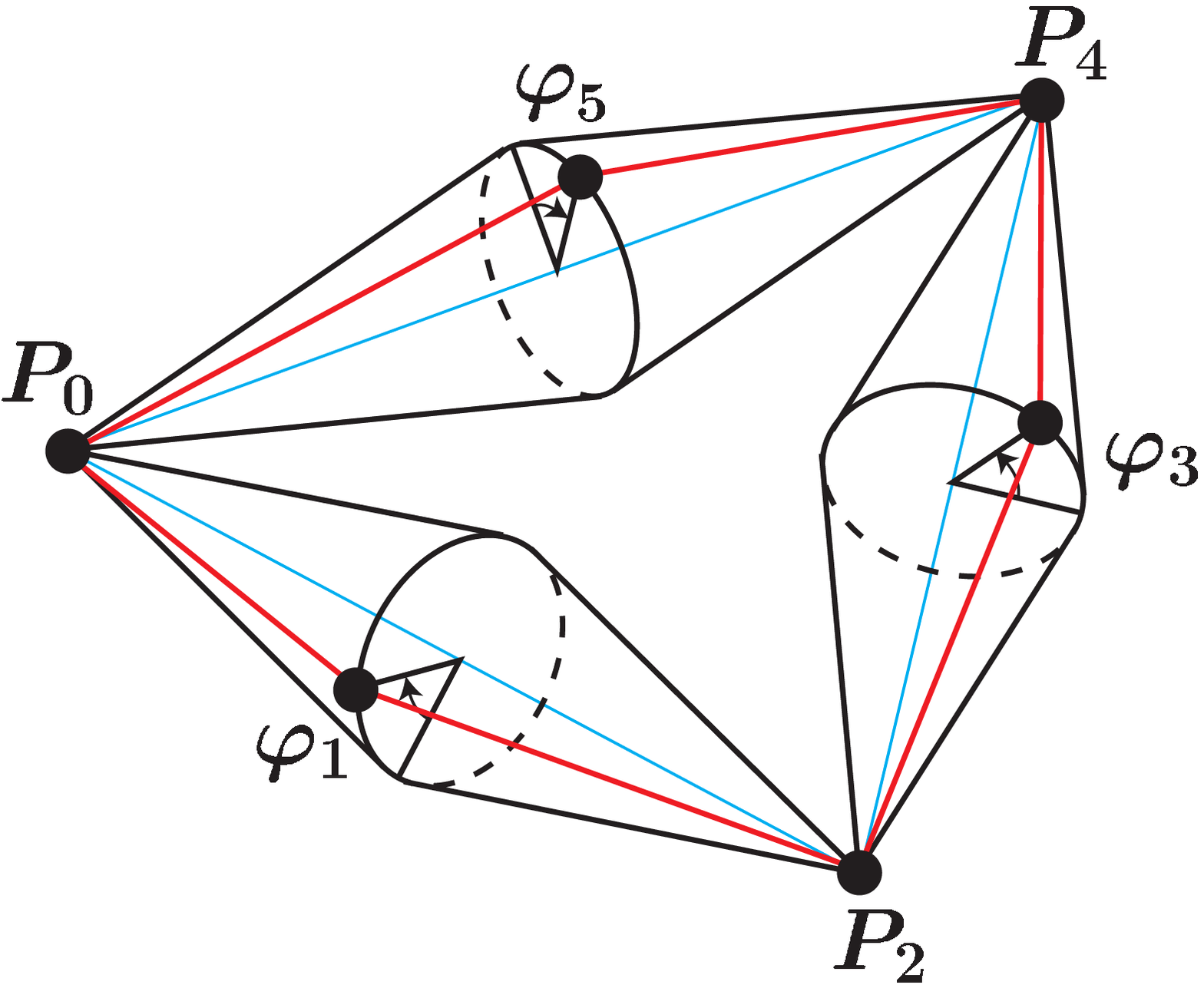}
\caption{Double cone or suspension expression}
\label{double-cone}
\end{center}
\end{figure}
\begin{equation}\label{P135}
P_1=\left(\!\begin{array}{c}
0 \\[1mm] -C \cos\vf_1 \\[1mm] \phantom{-}C\sin\vf_1
\end{array}\!\right), \>
P_3=\left(\!\begin{array}{c}
\frac{1}2S+\frac{\sqrt3}2C\cos\vf_3 \\[1mm] \frac{\sqrt3}2S+\frac{1}2C\cos\vf_3 \\[1mm] C\sin\vf_3
\end{array}\!\right), \>
P_5=\left(\!\begin{array}{c}
-\frac{1}2S-\frac{\sqrt3}2C\cos\vf_5\\[1mm] \frac{\sqrt3}2S+\frac{1}2C\cos\vf_5 \\[1mm] C\sin\vf_5
\end{array}\!\right)
\end{equation}
for some $\vf_1, \vf_3$, and $\vf_5$. 
Now the conditions $|P_j-P_{j+1}|=1$ $(\forall j)$ and $\angle P_1=\angle P_3=\angle P_5=\theta$ are satisfied. 

\smallskip
The condition $\angle P_{i+1}=\theta$ $(i=1,3,5)$ is equivalent to 
\begin{equation}\label{angle_condition}
C^2(\cos\vf_i\cos\vf_{i+2}-2\sin\vf_i\sin\vf_{i+2})+\sqrt3\,SC\,(\cos\vf_i+\cos\vf_{i+2})=3-5C^2,
\end{equation}
which is equivalent to 
\begin{equation}\label{angle_condition2}
C\big(C\cos\vf_i+\sqrt3\,S\,\big)\cos\vf_{i+2}-2C^2\sin\vf_i\sin\vf_{i+2}=3-5C^2-\sqrt3\,SC\,\cos\vf_i.
\end{equation}
Remark that the equations $ax+by=d$ $(a^2+b^2>0)$ and $x^2+y^2=1$ have solutions if and only if $a^2+b^2-d^2\ge0$, when we have 
\begin{equation}\label{xy}
x=\frac{ad\pm b\sqrt{a^2+b^2-d^2}}{a^2+b^2}, \> y=\frac{bd\mp a\sqrt{a^2+b^2-d^2}}{a^2+b^2}.
\end{equation}
In our case \eqref{angle_condition2}, by substituting 
\begin{equation}\label{abd}
a=C\big(C\cos\vf_i+\sqrt3\,S\,\big),\> b=-2C^2\sin\vf_i, \>d=3-5C^2-\sqrt3\,SC\,\cos\vf_i,\end{equation}
we have $a^2+b^2=4-{(S-\sqrt3\,C\,\cos\vf_i)}^2$, which is positive unless $\theta=\frac\pi3$ and $\vf_i=\pi$, and 
\begin{equation}\label{a^2+b^2-d^2}
a^2+b^2-d^2=-\sqrt3\big(C\,\cos\vf_i-\sqrt3S\,\big)\big(\sqrt3\,C\,\cos\vf_i-(3-8C^2)S\,\big). 
\end{equation}

It follows that when $(\theta,\vf_1)\ne(\frac\pi3,\pi)$ there are $\vf_3$ and $\vf_5$ so that $\angle P_0=\angle P_2=\theta$ if and only if $\cos\vf_1$ satisfies 
$$
\frac{(3-8C^2)S}{\sqrt3\,C}\le\cos\vf_1\le\frac{\sqrt3\,S}C,
$$
which can happen if and only if $C=\cos\frac\theta2\ge\frac12$, i.e. $0\le\theta\le\frac{2\pi}3$ {\rm (Figure )}.

\begin{figure}[htbp]
\begin{center}
\includegraphics[width=.42\linewidth]{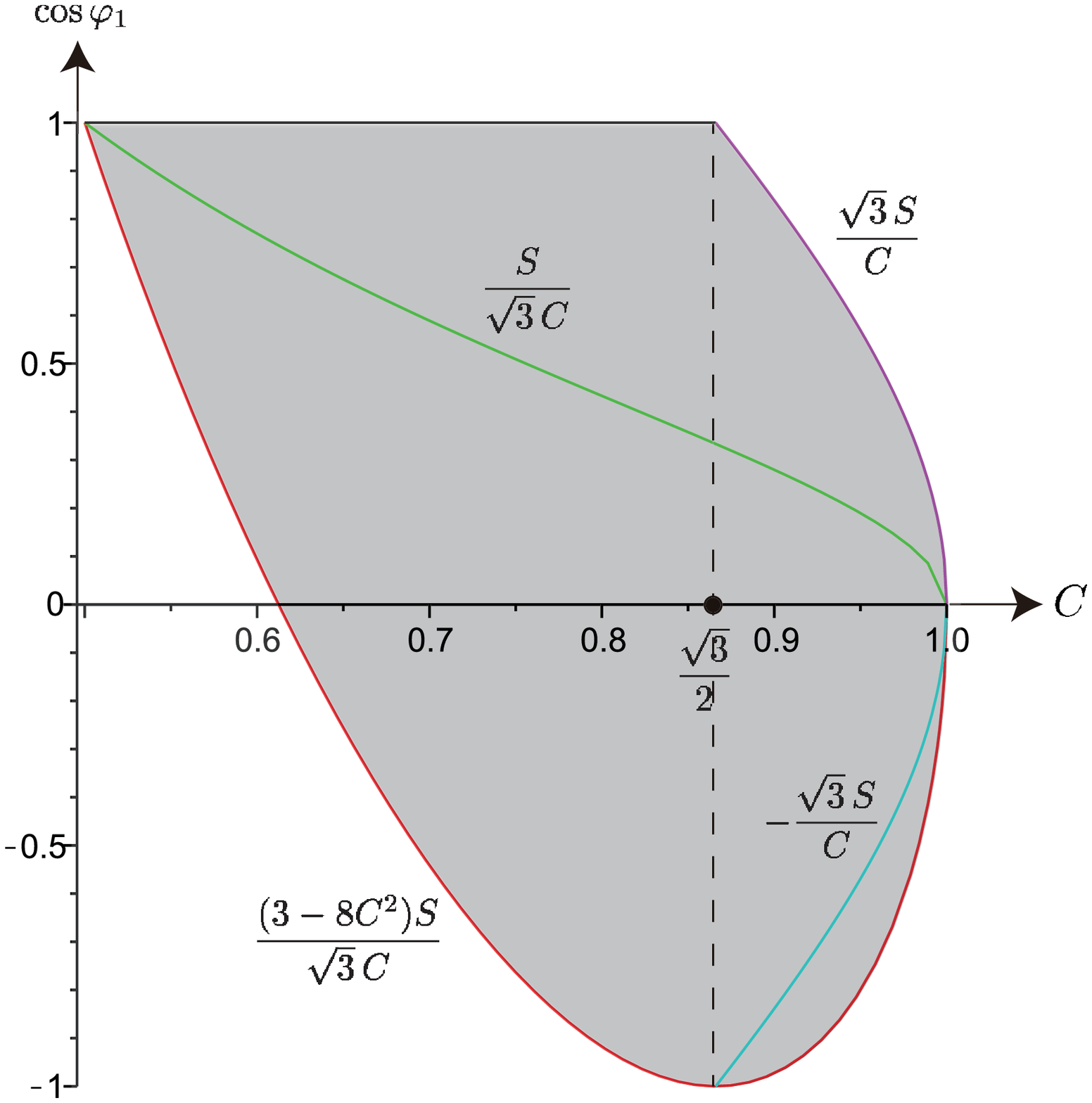}
\caption{The region of $\cos\vf_1$ so that there are $\vf_3$ and $\vf_5$ satisfying $\angle P_0=\angle P_2=\theta$}
\label{graph}
\end{center}
\end{figure}

\smallskip
(1) Let us first assume that $\theta$ and $\vf_1$ satisfy the conditions above mentioned 
and search for the case when $\vf_3$ and $\vf_5$ that make $\angle P_0=\angle P_2=\theta$ also satisfy $\angle P_4=\theta$. 
(We will study the case when $(\theta,\vf_1)=(\frac\pi3,\pi)$ later.) 
We have two cases, either $\vf_3\ne\vf_5$ or $\vf_3=\vf_5$. 

\smallskip
Case I. 
Assume $\vf_3\ne\vf_5$. 
Remark that this can occur if and only if $\vf_1$ satisfies 
$$
\frac{(3-8C^2)S}{\sqrt3\,C}<\cos\vf_1<\frac{\sqrt3\,S}C \hspace{0.5cm}\left(0<\theta<\frac{2\pi}3\right). 
$$
Then the conditions $\angle P_0=\angle P_2=\theta$ imply that $\vf_3$ and $\vf_5$ are given by 
\begin{equation}\label{vf3nevf5}
\{(\cos\vf_3,\sin\vf_3),(\cos\vf_5,\sin\vf_5)\}=\left\{\left(\frac{ad\pm b\sqrt{a^2+b^2-d^2}}{a^2+b^2},\,\frac{bd\mp a\sqrt{a^2+b^2-d^2}}{a^2+b^2}\right)\right\},
\end{equation}
where $a,b$, and $d$ are given by \eqref{abd}. 
Computating the left hand side of \eqref{angle_condition}, we have 
$$
C^2\left(\frac{a^2d^2-b^2(a^2+b^2-d^2)}{{(a^2+b^2)}^2}-2\,\frac{b^2d^2-a^2(a^2+b^2-d^2)}{{(a^2+b^2)}^2}\right)
+\sqrt3\,SC\,\frac{2ad}{a^2+b^2}=3-5C^2, 
$$
which implies that the condition $\angle P_4=\theta$ is always satisfied in this case. 
Now \eqref{P135} shows that $P_3$ and $P_5$ are given by 
\begin{equation}\label{P35}
\begin{array}{l}
\displaystyle 
P_3=\left(\begin{array}{c}
\displaystyle S-\frac{\sqrt3\,C\,\cos\vf_1-(3-8C^2)S\pm\sqrt3\,C\,\sin\vf_1\sqrt{a^2+b^2-d^2}}{4-{(S-\sqrt3\,C\,\cos\vf_i)}^2}\\[4mm]
\displaystyle \frac2{\sqrt3}S-\frac{C\,\cos\vf_1-\frac1{\sqrt3}(3-8C^2)S\pm C\,\sin\vf_1\sqrt{a^2+b^2-d^2}}{4-{(S-\sqrt3\,C\,\cos\vf_i)}^2}\\[4mm]
\displaystyle -\frac{2C\big(3-5C^2-\sqrt3\,SC\,\cos\vf_1\big)\sin\vf_1\pm\big(C\cos\vf_1+\sqrt3\,S\,\big)\sqrt{a^2+b^2-d^2}}{4-{(S-\sqrt3\,C\,\cos\vf_i)}^2}
\end{array}\right) \\[18mm]
\displaystyle 
P_5=\left(\begin{array}{c}
\displaystyle -S+\frac{\sqrt3\,C\,\cos\vf_1-(3-8C^2)S\mp\sqrt3\,C\,\sin\vf_1\sqrt{a^2+b^2-d^2}}{4-{(S-\sqrt3\,C\,\cos\vf_i)}^2}\\[4mm]
\displaystyle \frac2{\sqrt3}S-\frac{C\,\cos\vf_1-\frac1{\sqrt3}(3-8C^2)S\mp C\,\sin\vf_1\sqrt{a^2+b^2-d^2}}{4-{(S-\sqrt3\,C\,\cos\vf_i)}^2}\\[4mm]
\displaystyle -\frac{2C\big(3-5C^2-\sqrt3\,SC\,\cos\vf_1\big)\sin\vf_1\mp\big(C\cos\vf_1+\sqrt3\,S\,\big)\sqrt{a^2+b^2-d^2}}{4-{(S-\sqrt3\,C\,\cos\vf_i)}^2}
\end{array}\right) ,
\end{array}
\end{equation}
where $a^2+b^2-d^2$ is given by \eqref{a^2+b^2-d^2}.

\smallskip
Case II. Assume $\vf_3=\vf_5$. 
Let us first study the condition for $\angle P_4=\theta$ without assuming $\angle P_0=\angle P_2=\theta$. 
If $\vf_3=\vf_5$, which we denote by $\vf$, and $\angle P_4=\theta$, then \eqref{angle_condition} implies that $\vf$ must satisfy 
$$
\cos\vf=-\frac{\sqrt3\, S}C \>\>\mbox{ or }\>\> \frac{S}{\sqrt3\,C}.
$$

Case II-1. Assume 
$$
(\cos\vf_3, \sin\vf_3)=(\cos\vf_5, \sin\vf_5)=
\left(\frac{S}{\sqrt3\,C}\,, \,\frac{\sqrt{4C^2-1}}{\sqrt3\,C}\right). 
$$
Then \eqref{angle_condition} implies that $\angle P_0=\angle P_2=\theta$ if and only if 
$$
(\cos\vf_1, \sin\vf_1)=
\left(\frac{S}{\sqrt3\,C}\,, \, \frac{\sqrt{4C^2-1}}{\sqrt3\,C}\right)
\>\>\mbox{ or }\>\>
\left(\frac{(3-8C^2)S}{\sqrt3\,C}\,, \, \frac{(4C^2-3)\sqrt{4C^2-1}}{\sqrt3\,C}\right).
$$
Note that we have $\overrightarrow{P_0P_5}=\overrightarrow{P_2P_3}$ by \eqref{P135}. 
The points $P_1$ and $P_4$ are in the opposite (or same) side of the plane containing $P_0, P_2, P_3$, and $P_5$ if $\cos\vf_1=\frac{S}{\sqrt3\,C}$ (or respectively $\cos\vf_1=\frac{(3-8C^2)S}{\sqrt3\,C}$). 
Namely, the hexagon is a chair if $\vf_1=\vf_3=\vf_5$ and a boat if $\vf_1\ne\vf_3=\vf_5$ in this case. 
Both coincide if and only if $\theta=0$ or $\frac{2\pi}3$, when $\mathcal{P}$ is a $6$-times covered multiple edge or a regular hexagon. 
The chair is given by 
$$
P_1=\left(\!\begin{array}{c}
0 \\[1mm]  -\frac{S}{\sqrt3} \\[2mm]  \frac{\sqrt{4C^2-1}}{\sqrt3}
\end{array}\!\right), \>
P_3=\left(\!\begin{array}{c}
S \\[1mm]  \frac{2S}{\sqrt3} \\[2mm]  \frac{\sqrt{4C^2-1}}{\sqrt3}
\end{array}\!\right), \>
P_5=\left(\!\begin{array}{c}
-S \\[1mm]  \frac{2S}{\sqrt3} \\[2mm]  \frac{\sqrt{4C^2-1}}{\sqrt3}
\end{array}\!\right),
$$
whereas the boat is given by substituting $(\cos\vf_1, \sin\vf_1)=\left(\frac{(3-8C^2)S}{\sqrt3\,C}\,, \, \frac{(4C^2-3)\sqrt{4C^2-1}}{\sqrt3\,C}\right)$ to \eqref{P35}, 
$$
P_1=\left(\!\begin{array}{c}
0 \\[1mm]  -\frac{(3-8C^2)S}{\sqrt3} \\[2mm]  \frac{(4C^2-3)\sqrt{4C^2-1}}{\sqrt3}
\end{array}\!\right), \>
P_3=\left(\!\begin{array}{c}
S \\[1mm]  \frac{2S}{\sqrt3} \\[2mm]  \frac{\sqrt{4C^2-1}}{\sqrt3}
\end{array}\!\right), \>
P_5=\left(\!\begin{array}{c}
-S \\[1mm]  \frac{2S}{\sqrt3} \\[2mm]  \frac{\sqrt{4C^2-1}}{\sqrt3}
\end{array}\!\right).
$$

\smallskip
Case II-2. Assume 
$$
(\cos\vf_3, \sin\vf_3)=(\cos\vf_5, \sin\vf_5)=
\left(-\frac{\sqrt3\, S}C\,, \,\frac{\sqrt{4C^2-3}}{C}\right), 
$$
which can occur if and only if $\frac{\sqrt3}2\le C\le 1$, namely, $0\le\theta\le\frac\pi3$. 
Then \eqref{angle_condition} implies that $\angle P_0=\angle P_2=\theta$ if and only if 
$
2C\sqrt{4C^2-3}\,\sin\vf_1=8C^2-6.
$
Therefore, when $\theta\ne\frac\pi3$ (we will study the case when $\theta=\frac\pi3$ and $\vf_3=\vf_5=\pi$ later) then $\angle P_0=\angle P_2=\theta$ if and only if 
$$
(\cos\vf_1, \sin\vf_1)=
\left(-\frac{\sqrt3\, S}C\,, \,\frac{\sqrt{4C^2-3}}{C}\right)
\>\>\mbox{ or }\>\>
\left(\frac{\sqrt3\, S}C\,, \,\frac{\sqrt{4C^2-3}}{C}\right).
$$
Note that $P_3$ and $P_5$ are above $P_0$ and $P_2$ respectively. 
When $\vf_1=\vf_3=\vf_5$ the hexagon is an ``inward crown'' (Figure \ref{cylinder}) given by 
$$
P_1=\left(\!\begin{array}{c}
0 \\[1mm]  \sqrt3\,S \\[2mm]  \sqrt{4C^2-3}
\end{array}\!\right), \>
P_3=\left(\!\begin{array}{c}
-S \\[1mm]  0 \\[2mm]  \sqrt{4C^2-3}
\end{array}\!\right), \>
P_5=\left(\!\begin{array}{c}
S \\[1mm]  0\\[2mm]  \sqrt{4C^2-3}
\end{array}\!\right),
$$
whereas the other is given by substituting $(\cos\vf_1, \sin\vf_1)=\left(\frac{\sqrt3\, S}C\,, \,\frac{\sqrt{4C^2-3}}{C}\right)$ to \eqref{P35}, 
$$
P_1=\left(\!\begin{array}{c}
0 \\[1mm]  -\sqrt3\,S \\[2mm]  \sqrt{4C^2-3}
\end{array}\!\right), \>
P_3=\left(\!\begin{array}{c}
-S \\[1mm]  0 \\[2mm]  \sqrt{4C^2-3}
\end{array}\!\right), \>
P_5=\left(\!\begin{array}{c}
S \\[1mm]  0\\[2mm]  \sqrt{4C^2-3}
\end{array}\!\right).
$$

\smallskip
Let us summarize the argument above when $\theta\ne\frac\pi3$. 
\begin{theorem} Suppose a $\theta$-equiangular unit equilateral hexagon $(\theta\ne\frac\pi3)$ is parametrized by the agnles $\vf_1, \vf_3$, and $\vf_5$ by \eqref{P024}, \eqref{P135}. Let $C=\cos\left(\frac\theta2\right)$ and $S=\sin\left(\frac\theta2\right)$ as before. 

\begin{enumerate}
\item When $\theta=\frac{2\pi}3$ i.e. $C=\frac12$ we have $\vf_1=\vf_3=\vf_5=0$, which corresponds to a regular hexagon. 
\item When $\frac{\pi}3<\theta<\frac{2\pi}3$ i.e. $\frac12<C<\frac{\sqrt3}2$ we have $-\arccos\left(\frac{(3-8C^2)S}{\sqrt3\,C}\right)\le\vf_1\le\arccos\left(\frac{(3-8C^2)S}{\sqrt3\,C}\right)$. 
\begin{itemize}
\item[\rm (i)] When $\vf_1=\pm\arccos\left(\frac{(3-8C^2)S}{\sqrt3\,C}\right)$ we have $\vf_3=\vf_5=\mp\arccos\left(\frac{S}{\sqrt3\,C}\right)$, which corresponds to a boat (Figure \ref{boat6}). 
\item[\rm (ii)] When $-\arccos\left(\frac{(3-8C^2)S}{\sqrt3\,C}\right)<\vf_1<\arccos\left(\frac{(3-8C^2)S}{\sqrt3\,C}\right)$ we have either 
\begin{itemize}
\item[\rm *] $\vf_3\ne\vf_5$, which are given by \eqref{vf3nevf5}, 
\item[\rm *] $\vf_1=\vf_3=\vf_5=\pm\arccos\left(\frac{S}{\sqrt3\,C}\right)$, which corresponds to a chair (Figure \ref{chair6}). 
\end{itemize}
\end{itemize}
\item When $0<\theta<\frac{\pi}3$ i.e. $\frac{\sqrt3}2<C<1$ we have $\arccos\left(\frac{\sqrt3\,S}C\right)\le|\vf_1|\le\arccos\left(\frac{(3-8C^2)S}{\sqrt3\,C}\right)$. 
\begin{itemize}
\item[\rm (i)] When $\vf_1=\pm\arccos\left(\frac{(3-8C^2)S}{\sqrt3\,C}\right)$ we have $\vf_3=\vf_5=\pm\arccos\left(\frac{S}{\sqrt3\,C}\right)$, which corresponds to a boat (Figure \ref{IMG_2796}). 
\item[\rm (ii)] When $\vf_1=\pm\arccos\left(\frac{\sqrt3\,S}C\right)$ we have $\vf_3=\vf_5=\pm\arccos\left(-\frac{\sqrt3\,S}C\right)=\pm\left(\pi-\arccos\left(\frac{\sqrt3\,S}C\right)\right)$.
\item[\rm (iii)] When $\arccos\left(\frac{\sqrt3\,S}C\right)<|\vf_1|<\arccos\left(\frac{(3-8C^2)S}{\sqrt3\,C}\right)$ we have either 
\begin{itemize}
\item[\rm *] $\vf_3\ne\vf_5$, which are given by \eqref{vf3nevf5}, 
\item[\rm *] $\vf_1=\vf_3=\vf_5=\pm\arccos\left(\frac{S}{\sqrt3\,C}\right)$, which corresponds to a chair (Figure \ref{udchair6}). 
\item[\rm *] $\vf_1=\vf_3=\vf_5=\pm\arccos\left(-\frac{\sqrt3\,S}C\right)$, which corresponds to an ``inward crown'' (Figure \ref{cylinder}). 
\end{itemize}
\end{itemize}
\item When $\theta=0$ the hexagon degenerates to a $6$ times covered multiple edge. 
\end{enumerate}
\end{theorem}
%
\begin{figure}[htbp]
\begin{center}
\begin{minipage}{.45\linewidth}
\begin{center}
\includegraphics[width=.5\linewidth]{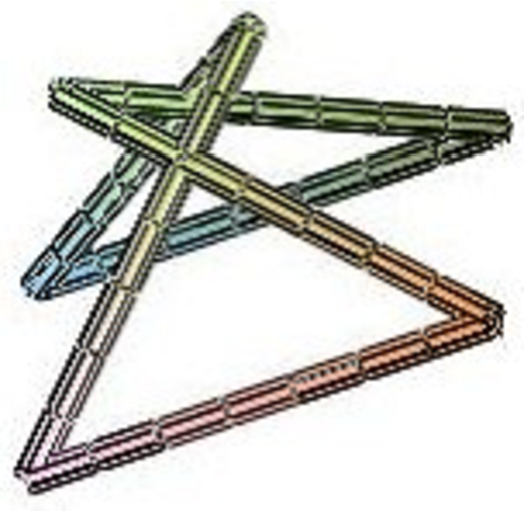}
\caption{A boat with a small bond angle}
\label{IMG_2796}
\end{center}
\end{minipage}
\hskip 0.4cm
\begin{minipage}{.45\linewidth}
\begin{center}
\includegraphics[width=0.33\linewidth]{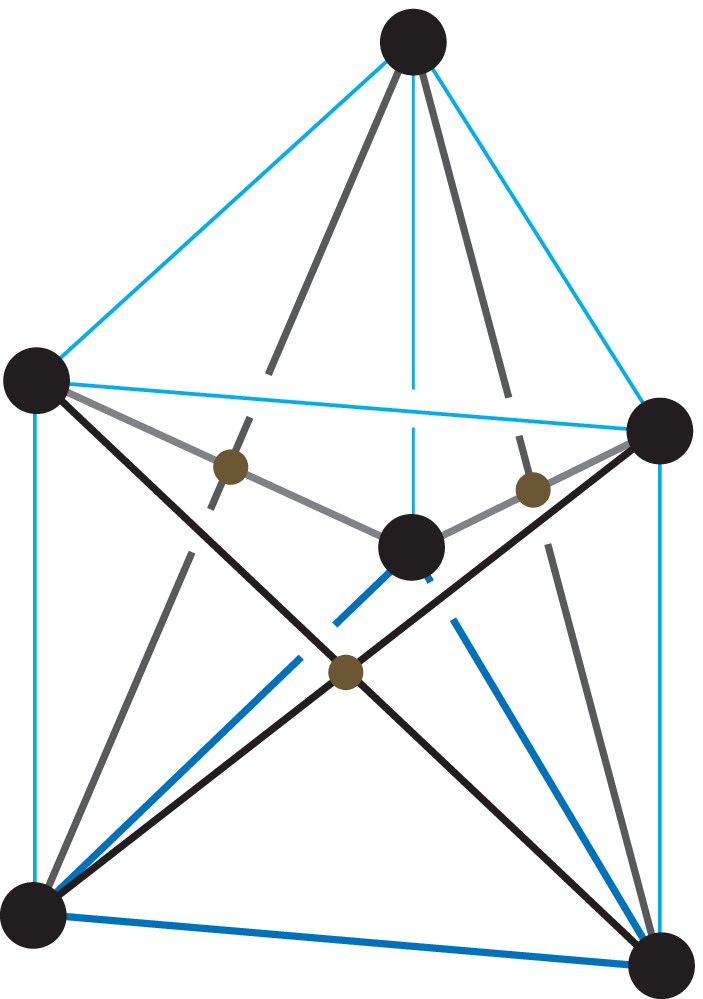}
\caption{``{\sl inward crown}'' in a prism. Two edges intersect each other in a side face of the prism.}
\label{cylinder}
\end{center}
\end{minipage}
\end{center}
\end{figure}
%
\begin{corollary}
The configuration space $\mathcal{M}_6(\theta)$ of $\theta$-equiangular unit equilateral hexagons $(\theta\ne\frac\pi3)$ is homeomorphic to a point if $\theta=0,\frac{2\pi}3$, the union of a circle and a pair of points if $\frac{\pi}3<\theta<\frac{2\pi}3$, the union of two circles and four points if $0<\theta<\frac{\pi}3$, and the empty set if $\theta<0$ or $\theta>\frac{2\pi}3$. 

Boat configurations are included in circles above mentioned, and chairs are isolated. 
\end{corollary}

We will see that a boat degenerates to a planar configuration when $\theta=\frac\pi3$. 

\begin{corollary} \begin{enumerate}
\item A boat and its mirror image can be joined by a path in the configuration space, i.e. they can be continuously deformed from one to the other, if and only if the bond angle satisfies $\frac{\pi}3<\theta<\frac{2\pi}3$. 
\item A boat and a chair cannot be joined by a path in the configuration space, i.e. they cannot be continuously deformed from one to the other. 
\end{enumerate}
\end{corollary}

\smallskip
(2) Finally we study the exceptional case $\theta=\frac\pi3$, when the cases when $\vf_j=\pi$ $(j=1,3,5)$ have not been considered yet. 

When $\theta=\frac\pi3$ the equation \eqref{angle_condition} becomes 
$$
(\cos\vf_i+1)(\cos\vf_{i+2}+1)-2\sin\vf_i\sin\vf_{i+2}=0.
$$
If $\angle P_2=\frac\pi3$ then $\vf_3$ is determined by $\vf_1$ as follows; 
\begin{itemize}
\item if $\vf_1=\pi$ then $\vf_3$ is arbitrary, 
\item if $\vf_1=0$ then $\vf_3=\pi$, 
\item if $\vf_1\ne0,\pi$ then $\vf_3=\pi$ or $f(\vf_1)$ ($f(\vf_1)\ne\pi$), where $f(\vf)$ ($\vf\ne\pm\pi$) is given by 
\begin{equation}\label{f(phi)}
(\cos f(\vf), \,\sin f(\vf))=\left(\frac{-(\cos\vf+1)^2+4\sin^2\vf}{(\cos\vf+1)^2+4\sin^2\vf},\,\frac{4\sin\vf (\cos\vf+1)}{(\cos\vf+1)^2+4\sin^2\vf}\right).
\end{equation}
\end{itemize}
Remark that $f(0)=\pi$ and that $f(\vf)=\vf$ if and only if $\vf=\pm\arccos(\frac13)$. 
Put $f(\pi)=0$ as $\displaystyle \lim_{\vf\to\pi}f(\vf)=0$. 
\begin{theorem}
Suppose a $\frac\pi3$-equiangular unit equilateral hexagon is parametrized by the agnles $\vf_1, \vf_3$, and $\vf_5$ by \eqref{P024}, \eqref{P135}. Then we have
$$
\{\vf_1,\vf_3,\vf_5\}=\{\pi,\vf,f(\vf)\},\>\{\pi,\pi,\vf\},\>\>\mbox{ or }\>\>\left\{\pm\arccos\left(\frac13\right),\pm\arccos\left(\frac13\right),\pm\arccos\left(\frac13\right)\right\}, 
$$
where $\vf$ is arbitrary. The first case contains a boat when $\{\vf_1,\vf_3,\vf_5\}=\{\pi,\pm\arccos\left(\frac13\right),\pm\arccos\left(\frac13\right)\}$, and the last triples correspond to a chair. 
\end{theorem}

\begin{corollary}
The configuration space $\mathcal{M}_6(\frac\pi3)$ of equilateral and $\frac\pi3$-equiangular hexagons is homeomorphic to the union of a pair of points and the space $X$ illustrated in Figure \ref{singular_hexagons} which is a $1$-skelton of a tetrahedron with edges being doubled. 
\end{corollary}
\begin{figure}[htbp]
\begin{center}
\includegraphics[width=.8\linewidth]{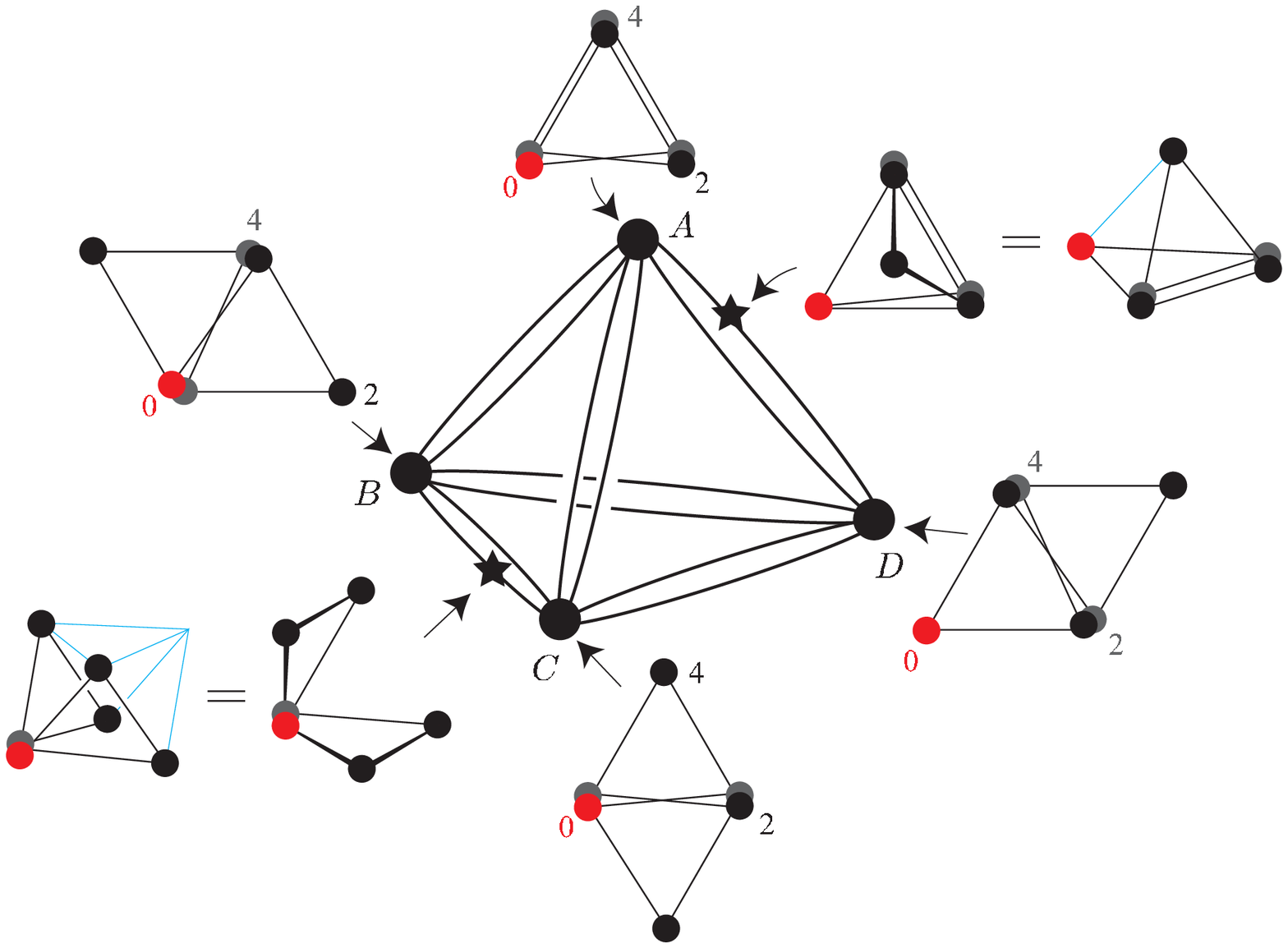}
\caption{The configuration space $X=\mathcal{M}_6(\frac{\pi}3)\setminus\{\rm chairs\}$. The numbers $0,2$, and $4$ in the figure indicate the vertices $P_0,P_2$, and $P_4$ respectively. 
The figures of six hexagons around $X$ are seen from above. The non-planar configurations left below is a boat, when $P_i$ occupy five vertices of a regular octahedron. \newline
\indent
 The four vertices $A,B,C$, and $D$ of $X$ correspond to planar configurations parametrized by $(\vf_1,\vf_3,\vf_5)=(\pi,\pi,\pi),\,(\pi,\pi,0),\,(0,\pi,\pi)$, and $(\pi,0,\pi)$ respectively. 
%
The circle through $A$ and $D$ consists of the configurations 
parametrized by $(\vf_1,\vf_3,\vf_5)=(\pi,\vf,\pi)$ $(-\pi\le\vf\le\pi)$. 
%
The circle hrough $B$ and $C$ consists of the configurations 
parametrized by $(\vf_1,\vf_3,\vf_5)=(\vf,\pi,f(\vf))$ $(-\pi\le\vf\le\pi)$. 
}
\label{singular_hexagons}
\end{center}
\end{figure}

\bigskip
The author would like to close the article with an open problem: find a new invarian which can show that a boat cannot be deformed continuously into a chair.

\bibliographystyle{plain}

\bigskip \noindent
Department of Mathematics and Information Sciences, \\
Tokyo Metropolitan University, \\
1-1 Minami-Ohsawa, Hachiouji-Shi, Tokyo 192-0397, JAPAN. \\
E-mail: ohara@tmu.ac.jp\\
Fax: 81-42-677-2481

\end{document}